\documentclass[10pt,reqno]{amsart}
\usepackage{amsfonts}
\usepackage{amssymb}
\usepackage{graphicx}
\usepackage{newlfont}
\usepackage{mathrsfs}
\usepackage{graphicx}
\usepackage{amsmath, amssymb, amsfonts,amsthm}
\usepackage{textcomp}
\usepackage{color}
\usepackage{mathtools}
\usepackage{multicol}
\usepackage{bbm}
\usepackage{float}
\usepackage{enumerate}
\usepackage{todonotes} 

\usepackage{hyperref}

\usepackage[numbers,sort&compress]{natbib}

\newtheorem {thm}{Theorem}
\newtheorem{theorem}{Theorem}[section]
\newtheorem{cor}[theorem]{Corollary}
\newtheorem{lem}[theorem]{Lemma}
\theoremstyle{remark}
\newtheorem{rem}[theorem]{Remark}
\newtheorem{prop}[theorem]{Proposition}
\newtheorem{definition}[theorem]{Definition}

\newcommand{\HH}{\mbox{${\mathcal H}$}}

\newcommand{\OO}{\mbox{${\mathcal O}$}}

\newcommand{\Rbold}{\mbox{${\mathbb R}$}}

\newcommand{\Ebold}{\mbox{${\mathbb E}$}}
\newcommand{\Pbold}{\mbox{${\mathbb P}$}}

\newcommand{\ee}{\mbox{$\overrightarrow{e}$}}

\def \a   {\alpha}
\def \k   {\kappa}

\newcommand{\eps}{\epsilon}

\def\be#1\ee{\begin{equation}#1\end{equation}}
\def\ben#1\een{\begin{equation*}#1\end{equation*}}

\begin{document}

\title[Positive reinforced generalized time-dependent P\'olya urns]{Positive reinforced generalized time-dependent P\'olya urns via stochastic approximation}
\author{Wioletta M. Ruszel} 
\address[Utrecht University]{Mathematical Institute}
\address{Budapestlaan 6, 3584 CD Utrecht, The Netherlands}
\email{w.m.ruszel@uu.nl}          
\author{Debleena Thacker}  
\address[Durham University]{Department of Mathematical Sciences}
\address{Campus, Stockton Rd, Durham DH1 3LE, United Kingdom}         
\email{debleena.thacker@durham.ac.uk} 
        
\date{\today}

\begin{abstract}
Consider a generalized time-dependent P\'olya urn process defined as follows. Let $d\in \mathbb{N}$ be the number of urns/colors. At each time $n$, we distribute $\sigma_n$ balls randomly to the $d$ urns, proportionally to $f$, where $f$ is a valid reinforcement function. We consider a general class of positive reinforcement functions $\mathcal{R}$ assuming some monotonicity and growth condition. The class $\mathcal{R}$ includes convex functions and the classical case $f(x)=x^{\alpha}$, $\alpha>1$. The novelty of the paper lies in  extending stochastic approximation techniques to the $d$-dimensional case and proving that eventually the process will fixate at some random urn and the other urns will not receive any balls any more.
\end{abstract}


                                             
\maketitle

\section{Introduction}
\label{Sec: Intro}
The classical P\'olya urn model with two urns (colors) to which balls are added randomly was introduced by  Eggenberger and P\'olya in 1923, \cite{Polya}. 
Since then, many generalizations and extensions of the classical model have been studied, see e.g. \cite{Peman} for a survey. One of the fundamental questions is how the composition of the urns will look like and how it depends on the way balls are added as time goes to infinity.
There are numerous applications in economics, computer science and biology where the model is better known as \textit{balls and bins model}, \cite{trade, trade1, trade2, CS, bio, bai} to mention a few.

A popular generalization is the \textit{non-linear P\'olya urn} model  or  \textit{balls and bin model with feedback}. The probability of a new ball choosing a bin with $x$ existing balls is proportional to $f(x)$ where $f$ will be referred to as the feedback function, \cite{trade2}. A common choice of the feedback function is $f(x) = x^{\alpha}$, $\alpha>0$. 

For the original case $\alpha=1$ , it was proven in \cite{Polya} that the proportion of balls in each bin converges to a beta-distributed random variable. In the positive feedback regime $\alpha>1$, which is also referred to as preferential attachment, the authors in \cite{trade2} proved  \textit{dominance}, i.e. almost surely the proportion of each bin converges to a $\{0,1\}$-random variable.  A stronger result which we call \textit{fixation} or \textit{monopoly} was proven by \cite{mon}. Fixation refers to the event that eventually one bin receives all but a finite number of balls. The onset of the time of fixation (speed of convergence towards the stationary distribution) in the positive feedback regime was studied in \cite{Oliviera_2009}. In fact, the author in \cite{Oliviera_2009} studies the onset of fixation for more general feedback functions satisfying some growth conditions and being perturbations of the canonical case $x^{\alpha}$. In particular for that class of feedback function we have that $\sum_{x=1}^{\infty} \frac{1}{f(x)} <\infty$. 

The \textit{negative feedback regime} is characterized by $\alpha<1$, and in this case the proportion of balls in the urns converges towards the uniform distribution on $\{1,2,\ldots,d\}$ where $d$ is the number of bins, \cite{trade2}. A time-dependent version of positively reinforced 2-urn models was studied in \cite{Si-2018, Pem-90} where at each time $n$, $\sigma_n$  balls (satisfying some growth condition) were randomly added to either one urn \cite{Pem-90} or independently to both \cite{Si-2018}. In both cases there will be dominance. Depending on some growth conditions of $\sigma_n$, \cite{Si-2018} demonstrates there might be no fixation.

Other generalizations include considering $d$ urns or colors, $d>2$ and more general addition rules like addition generating matrices and convex resp. concave feedback functions, \cite{bai1, bai2, LauPa1, LauPa}. The replacement matrix $H=\left(\left(H_{ij}\right)\right)_{ij}$ models the placement of $H_{ij}$ balls to urn $j$  when urn $i$ is chosen.
In \cite{LauPa} the authors prove that if the feedback function is strictly concave and $H$ bi-stochastic, then the urn composition converges towards the uniform distribution $(1/d,\cdots, 1/d)$. The limiting distribution will not be uniform with probability one if $f$ is convex. Finally they prove in the same paper that for $d=2$, $f$ convex and $H$ irreducible, the limiting proportion converges towards the equilibrium points of the corresponding mean-field function resulting from the stochastic approximation approach. A concave feedback function, which includes the negative reinforcement regime, tends to equalize the asymptotic distribution of the proportion of different colors, whereas a convex $f$ which includes the positive reinforcement regime tends to amplify the effect of the generating matrix $H$. In \cite{Kaur}, the author proves CLT type results for the proportion vector of colors around the uniform distribution in the negative reinforcement setting, when $H$ is double-stochastic and $f$ Lipschitz.

Urn models with infinitely many colors were treated in \cite{Deb1, Deb2, janson}. In \cite{Deb1} the authors
introduce a class of balanced urn schemes with infinitely many colors indexed by $\mathbb{Z}^d$ where the replacement schemes are given by the transition matrices associated with bounded increment random walks. They show that the urn composition of the $n$-th selected ball follows a Gaussian distribution. The authors in \cite{mai} generalized the possibly infinite space or urns to general Polish spaces and study the asymptotic behaviour of these measure-valued P\'olya urn processes. The author in \cite{janson} generalizes results obtained in \cite{Deb1} and \cite{mai} and studies measure-valued P\'olya urn processes under stochastic replacement matrices $H$.

A simplistic model for the reinforcement of neural connections in the brain using positive reinforced \textit{interacting P\'olya urns} was introduced in \cite{mark}. The urns/colors represent the edges of a graph. Roughly speaking, one  first chooses a random subset
 of colours (independent of the past) from $n$ colours of balls, and then positively reinforce a colour from this subset. In \cite{mark} the stability of equilibria and examples of different graphs were studied. Interesting follow-up research on percolation questions on different positively reinforced tree-like graphs and its application for neuronal connections were studied in \cite{hirsch, hirsch1}.\\

In this article we consider a generalized time-dependent P\'olya urn model with $d$ urns, $d\in \mathbb{N}$. More precisely, at each time step, $\sigma_n$ many balls are added randomly to the $d$ urns, $f$-proportionally to their weight with instantaneous replacement. We will assume that 
\[
(i) \sum^{\infty}_{n=1} \frac{\sigma_n}{\sum_{j=1}^n \sigma_j} = \infty \, \, \, \, \, \text{ and }  \, \, \, \, \, (ii) \sum^{\infty}_{n=1} \left (\frac{\sigma_n}{\sum_{j=1}^n \sigma_j} \right )^2< \infty.
\]
Condition (i) ensures that $\sigma_n\geq 1$ for all $n \in \mathbb{N}$ hence we keep on adding balls throughout the whole time-evolution, whereas condition (ii) restricts the growth of $\sigma_n$. Polynomial growth is allowed but not exponential growth, since it contradicts (ii). This condition is necessary to ensure that the  process, seen as a stochastic approximation (SA) scheme is not subjected to a large noise term which will hinder the process to converge. Both conditions appear naturally when using SA techniques from the dynamical systems viewpoint, see e.g. \cite{Rob} or \cite[Section 1]{Bor}. The class of reinforcement functions $\mathcal{R}$ (see Section \ref{Sec: Class of reinforcement functions} for its definition) is very general and satisfies some natural continuity and monotony conditions. W.l.o.g. we will evaluate $f$ not on the number of balls in an urn but on the proportions. Additionally, we assume that
\begin{equation}\label{eq:alpha}
\alpha := \inf_{x\in (0,1)} \frac{xf'(x)}{f(x)} >1.
\end{equation}
Note that if $f$ is convex or of the type $f(x)=x^{\alpha}$, $\alpha>1$ then \eqref{eq:alpha} is necessarily satisfied but the converse is not necessarily true.
A similar condition can be found in \cite{Oliviera_2009, LauPa}, in \cite{LauPa} the authors assume additionally that $f$ is convex and $H$ irreducible. Our case is not covered in \cite{LauPa}, since taking $H$ to be the identity matrix is not irreducible. We will prove in Theorems \ref{Thm: dominance} and \ref{Thm: fixation} that there is dominance and fixation for this class of general reinforcement functions $\mathcal{R}$ and generalized urn model. Extending ideas from \cite{Si-2018} we will prove that infinitely often the process will move away from non-trivial equilibrium points. Using SA techniques, \cite{Higham, Ben1,Ben2} and coupling the process to an appropriate ODE we will prove that the only possible stationary points of the process are the extremal points of the simplex $[0,1]^d$. In fact, 
condition \eqref{eq:alpha} ensures that the Jacobian of the ODE at all non-trivial points will have positive eigenvalues. This will imply that the dynamical system is not stable around non-trivial points. 
The novelty of the paper lies in proving dominance in this general setting where we consider $d$ urns, $d>2$, $f$ not necessarily convex, and the addition of balls $\sigma_n$ is variable. We extend SA techniques to this setting. To the best of our knowledge, SA techniques in the positive reinforcement regime were only applied to $d=2$ case in the literature, e.g. \cite{LauPa}.

\subsection{Outline of the paper}
\label{Sec: Outline of the paper}

The structure of the manuscript is as follows. In Section \ref{sec:def} we introduce the model and assumptions on the class reinforcement functions $f$. The results are presented in Section \ref{sec:results} whereas Section \ref{sec:proofs} is devoted to their proofs. Finally, in the Appendix \ref{sec:app} we introduce stochastic approximation techniques and relevant results.

\section{Model and Definitions}\label{sec:def}

\subsection{Model}
\label{Subsec: Model}
We assume that all random variables are defined on the same probability space $\left(\Omega, \mathcal{F}, \mathbb{P}\right)$.
We consider the following generalization of P\'olya urn scheme where
colors are indexed by a non-empty finite set $S:=\{1,2,\ldots,d\}$, $d\in \mathbb{N}$. For $n \in \mathbb{N}$ we denote the composition of the urn at time $n$ by $U_n:=\{U_{n,j}\}_{j \in S}$, where $U_{n,j}$ is the  "weight" of the $j-$th color at time $n$. 

We start with a non-trivial initial composition $U_0$, a given non-negative reinforcement function $f$ and a sequence of positive integers $\left(\sigma_n\right)_{n \ge 1}$. 
At every discrete time point $(n+1)$, $\sigma_{n+1}$ balls or colors are added $f$- proportional to their weight, with instantaneous replacement. That is, given $U_0, U_1, \ldots, U_n$, $\left(X_i^{(n+1)}\right)_{1 \leq i \leq \sigma_{n+1}}$ are i.i.d.\ random vectors, such that

\begin{equation}\label{Eq:choice_of_color}
\Pbold \left(X_i^{(n+1)}=e_j \mid U_0, U_1, \ldots U_n\right)= \frac{f (\theta_{n,j})}{\sum_{k=1}^d f (\theta_{n,k})},
\end{equation}
where $e_j, \text { } 1 \le j \le d$ is the unit vector corresponding to the canonical basis in $\Rbold^d$ and $\theta_n=\left(\theta_{n,k}\right)_{1 \le k \le d}$ is defined by $\theta_{n,k}=\frac{U_{n,k}}{\sum_{j=1}^d U_{n,j}}$. In other words, $\theta_{n,k}$ is the proportion of balls of color $k$ at time $n$. The urn composition is then updated according to the following rule:
\begin{equation}\label{Eq:Fundamental_equation}
U_{n+1}=U_n+ \displaystyle \sum_{i=1}^{\sigma_{n+1}} X_i^{(n+1)}.
\end{equation} 
In words, at time $(n+1)$ if the $j$-th color is selected at the $i$-th trial for $1\le i \le \sigma_{n+1}$, then we add a single ball of the same color to the urn.

Let us denote by $\tau_n$ the total number of balls at time $n$.
Observe that 
$$\tau_n=\|U_n\|_1= \tau_0 + \sum_{j=1}^{n+1} \sigma_j = \sum_{j=1}^d U_{0,j}+ \sum_{j=0}^{n-1} \sigma_j,$$ where $\| \cdot \|_1$ is the $\ell^1(\mathbb{N}^d)$-norm. Using this notation, $\theta_n= \frac{U_n}{\tau_n}$ and $\theta_n \in \{y \in \Rbold^d_+ \setminus\{0\}: \sum_{i=1}^d y_i=1\}$.

\begin{definition}\label{def:dom}
We call the event $\mathscr{D}$ \textit{dominance} if 
\begin{equation}
\mathscr{D} = \{ \exists i\in \{1,\cdots, d\} \text{ s.t. } \lim_{n\to \infty }(\theta_{n,1},\cdots \theta_{n,d})= e_i \}
\end{equation}
where $e_i$ is the $i$-th coordinate vector in $\Rbold^d$. 
\end{definition}
Dominance means that eventually the proportion of different colors in the urn becomes trivial, except for a single color.

\begin{definition}\label{def:fix}
We say the urn model $\left(U_n\right)_{n \ge 0}$ \textit{fixates} if almost surely, 
\begin{equation}\label{Def: fixation}
\mathscr{F} = \{ \exists N \ge 1 \text{ and } i\in \{1,\cdots, d\}\text{ s.t. for all }n \ge N, \, U_{n+1,i}= U_{n,i}+\sigma_{n+1}\}.
 \end{equation}
\end{definition}

It is important to note here that if the process fixates at some color $J$ out of $\{1, \cdots d\}$ (which is random), 
then all other colors stop growing after the random time of fixation. 
It is also clear that if the process fixates, it implies that there is almost surely dominance. 
The converse is not always true, (see \cite{Si-2018} for examples).

\subsubsection{Class of reinforcement functions $\mathcal{R}$}
\label{Sec: Class of reinforcement functions}
For the purpose of this paper, we will assume that $f:[0,1] \rightarrow \mathbb{R}_+$ satisfies 
\begin{itemize}
\item [{\bf(A)}] $f$ is a strictly non-decreasing and continuous, such that $f(0)=0$ and $f(1)=1$.
\item [{\bf(B)}] $f \in C^1((0,1))$ and the semi-derivatives $\lim_{x\rightarrow 0^+} f'(x)$ and $\lim_{x\rightarrow 1^-} f'(x)$ exist. 
\item [{\bf(C)}] 
\begin{equation}\label{Eq: Assumption_exponent}
\alpha:= \inf_{x \in \left(0,1\right)} \frac{x f'(x)}{f(x)} >1.
\end{equation}
\end{itemize}

The assumption {\bf(A)} that $f$ is non-decreasing, and $f(0)=0$ and $f(1)=1$ ia a natural assumption that ensures that "the higher the proportion of a color, more likely it is to be chosen". The assumption {\bf(B)} is for technical purposes.  Let $\mathcal{R}$ be the class of all functions that satisfy {\bf(A)}, {\bf(B)} and {\bf(C)}. 

Examples are:
\begin{itemize}
\item[(i)] $f(x)= x^{\a}, \text{ } \a >1$
\item[(ii)] $f(x)= x^{2+\epsilon} e^{-x +1}$, $\epsilon >0$
\end{itemize}
A classical example of such a function is example (i). From assumption {\bf(C)}, it may seem that for all $f \in \mathcal{R}$, $f$ has to be is convex. However, this is not the case as we will show in the following counterexample. Let $\epsilon>0$ and consider example (ii) Then $f$ satisfies trivially assumptions {\bf(A)} and {\bf(B)}. For the third assumption {\bf(C)}, note that
\[
f'(x) = (2+\epsilon - x) x^{1+\epsilon}e^{-x+1} 
\] 
so that $\inf_{x \in \left(0,1\right)} \frac{x f'(x)}{f(x)}=1+\epsilon >1$. This function is not convex, indeed
\[
f''(x) = (x^2 -2(2+\epsilon)x+(1+\epsilon)(\epsilon))x^{\epsilon} e^{-x+1}
\]
and we see easily see that for $\epsilon$ small, e.g. $\epsilon=0.1$, the second derivative is changing sign for $x\in (0,1)$. 
 
\begin{lem}\label{Lemma: properties_of_f} 
Let $f \in \mathcal{R}$. Then 
\begin{itemize}
\item[(i)] $f$ is Lipschitz on $(0,1)$.
\item[(ii)] \[
\lim_{x\rightarrow 0^+} \frac{f(x)}{x} = f'(0)  \ge 0,
\]
is well-defined.
\item[(iii)] The map $x\mapsto \frac{f(x)}{x}$ is increasing for all $x\in (0,1)$.
\item[(iv)] For all $x\in [0,1]$ we have that $f(x) \leq x^{\alpha}$, where $\alpha$ is defined in \eqref{Eq: Assumption_exponent}. 
\end{itemize} 
\end{lem}
\begin{proof}

(i) follows immediately from assumption {\bf{(B)}}. For (ii), observe that $f'(0)$ exists finitely is a part of the assumption {\bf{(B)}}, and $f'(0)  \ge 0$ follows from assumption {\bf(A)}.
The statement (iii) follows from the observation that derivative of $\frac{f(x)}{x}$ is given by
\ben
\frac{x f'(x)-f(x)}{x^2} \ge 0, 
\een
from assumption {\bf(C)}. Finally (iv) follows from  {\bf (C)} and 
\[
\log \left (\frac{f(1)}{f(x)} \right )= \int_{x}^{1}{\mathrm{d}(\log f(t))} \ge \alpha \log \left (\frac{1}{x} \right )
\]
for $x\in (0,1)$ and $\alpha$ was defined in \eqref{Eq: Assumption_exponent}.
\end{proof}
A major disadvantage of the class of functions $\mathcal{R}$ is that it does not include functions that decay exponentially, for example $f(x)= \frac{1}{a-1}(a^x -1)$ for $a >0$, as it fails to satisfy {\bf(C)} ($\alpha=1$ instead of $\alpha>1$). It is clear that Theorem \ref{Thm: dominance} should hold even for exponentially decaying functions, since similar result is shown in \cite{mon}, where $\sigma_n\equiv 1$. However, we could not apply the general stochastic approximation techniques for exponentially growing functions and for a general sequence $\left(\sigma_n\right)_{n \ge 0}$.

\subsubsection{Connections to Stochastic Approximation Theory}
In this section, we will connect the proportion vector $\left(\theta _n\right)_{n \ge 0}$ to a standard form of recursive equations of the stochastic approximations (SA) method, see e.g. \cite{Ben1,Ben2,Bor}. One of the most general forms of the recursion equation associated with SA is of the form

\begin{equation}\label{Eq: Standard_SA}
Y_{n+1}=Y_n+\gamma_{n+1}\left[H\left(Y_n, Z_n\right)+r_{n+1}\right], \text{ } Y_0 \in \Rbold^d, \text{ } n \ge 0,
\end{equation} 
where $\left(\gamma_n\right)_{n\geq 0}$ is the sequence of step sizes, 
$\left(Z_n\right)_{n\geq 0}$ is a sequence of i.i.d.\ random vectors and 
$\left(r_n\right)_{n \ge 0}$ is the sequence of "error" or remainder terms and $H:\Rbold^d \times \Rbold^d \rightarrow \Rbold^d$ is measurable function.   

A general practice is to re-write the above equation \eqref{Eq: Standard_SA} in the following form

\begin{equation}\label{Eq: Standard_SA_Martingale}
Y_{n+1}=Y_n+\gamma_{n+1}\left[\Ebold\left[H\left(Y_n, Z_n\right) | Y_0, Z_1, \ldots, Z_n\right]+ \Delta M_{n+1}+r_{n+1}\right],
\end{equation}

where $\Delta M_{n+1} = H\left(Y_n, Z_n\right)- \Ebold\left[H\left(Y_n, Z_n\right) | Y_0, Z_1, \ldots, Z_n\right]$. 

The advantage of such a representation is that under suitable conditions on $H$ and $\left(Z_n\right)_n$, one can relate the asymptotic properties of $Y_n$ to the zeros of the \textit{mean field function}, which we will define and discuss in details for our model. 

Recall our basic recursive equation \eqref{Eq:Fundamental_equation}. We can re-write this equation as follows 
\begin{equation}\label{Eq:fundamental_recursion_martingale_diff}
\theta_{n+1}= \theta_n+\frac{1}{\tau_{n+1}}\left[\Ebold\left[\sum_{i=1}^{\sigma_{n+1}} X_i^{(n+1)} \text{ }\Big \vert \mathcal{F}_n\right] -\sigma_{n+1} \theta _n+ \Delta M_{n+1} \right], 
\end{equation}
where $\mathcal{F}_n$ is the sigma algebra generated by $U_0, U_1. \ldots, U_n$, and 
\begin{equation}\label{def:Delta_Mn}
\Delta M_{n+1}= \sum_{i=1}^{\sigma_{n+1}} X_i^{(n+1)}- \Ebold\left[\sum_{i=1}^{\sigma_{n+1}} X_i^{(n+1)} \text{ }\Big \vert \mathcal{F}_n\right].
\end{equation}

Let us define the function $f_d : \Rbold^d \rightarrow \Rbold^d$ as follows $y \mapsto \left(f(y_1), f(y_2), \ldots, f(y_d)\right)$ for any vector $y=(y_1, y_2, \ldots, y_d ) \in \Rbold^d$ and $f\in \mathcal{R}$. It is easy to see that 
\begin{equation*}\label{Eq: conditional_expectation_for_choice_vector}
\Ebold\left[\sum_{i=1}^{\sigma_{n+1}} X_i^{(n+1)} \text{ }\Big \vert \mathcal{F}_n\right]=\sigma_{n+1} \frac{f_d(\theta_n)}{\|f_d (\theta_n) \|_1}.
\end{equation*}
An immediate consequence of the construction is the following corollary.

\begin{cor}
Given $\mathcal{F}_n$, the d-dimensional random vector $\sum_{i=1}^{\sigma_{n+1}} X_i^{(n+1)}$ follows a Multinomial distribution with parameters $\left(\sigma_{n+1}, \frac{f_d(\theta_n)}{\|f_d (\theta_n) \|_1}\right)$, i.e. for $x=(x_1,\cdots, x_d) \in \mathbb{N}^d$ such that $\sum_{i=1}^d x_i=\sigma_{n+1}$,
\[
\mathbb{P} \left ( \sum_{i=1}^{\sigma_{n+1}} X_i^{(n+1)} = x \right ) = \frac{\sigma_{n+1} !}{x_1!\ldots x_k!} \frac{f(\theta_{n, 1})^{x_1}}{\| f_d(\theta_n)\|_1} \cdots \frac{f(\theta_{n, d})^{x_d}}{\| f_d(\theta_n)\|_1}.
\]
\end{cor}

\noindent
Re-writing \eqref{Eq:fundamental_recursion_martingale_diff}, we have 
\begin{eqnarray}\label{Eq:fundamental_recursion_mean_field}
\nonumber \theta_{n+1} & = & \theta_n+\frac{1}{\tau_{n+1}}\left[ \sigma_{n+1} \frac{f_d(\theta_n)}{\|f_d (\theta_n) \|_1}-\sigma_{n+1} \theta _n+ \Delta M_{n+1}\right]\\
& =& \theta_n+\frac{\sigma_{n+1}}{\tau_{n+1}}\left[ h(\theta_n)+ \frac{1}{\sigma_{n+1}}\Delta M_{n+1}\right],
\end{eqnarray}
 where we define $h:\Rbold^d_+ \setminus \{0 \}\rightarrow \Rbold^d_+$ by $y \mapsto \left(\frac{f_d(y)}{\|f_d (y) \|_1}- y \right)$. This function $h$ is well-defined as $f$ is strictly non-decreasing and $f(0)=0$. $h$ will be the \textit{mean field function}, and we later see the relations of $h$ to the limit points of the sequence $\left( \theta _n\right)_{n \ge 0}$. Comparing \eqref{Eq:fundamental_recursion_mean_field} with \eqref{Eq: Standard_SA_Martingale}, we observe that it is not exactly in the form of the standard recursive equation as in practice in SA, due to the presence of the coefficient $\frac{1}{\sigma_{n+1}}$which is multiplied with the martingale difference $\Delta M_{n+1}$. We present all necessary results related to SA required for this paper in the Appendix \ref{sec:app}.

\section{Results}\label{sec:results}
In this section we will present our main results.
 
\begin{thm}\label{Thm: dominance}
Suppose that  $\left(\theta_n\right)_{n \ge 0}$ is as in \eqref{Eq:fundamental_recursion_mean_field}. Assume the following conditions:
\begin{itemize}
\item [(i)] Let $\left(\sigma_n\right)_{n \ge 1}$ be such that $\sum_{n \ge 1}\frac{\sigma_n}{\tau_n} = \infty$ and $\sum_{n \ge 1} \left (\frac{\sigma_n}{\tau_n} \right )^2 < \infty$,
\item[(ii)] $f \in \mathcal{R}$, 
\end{itemize}
then we have that
\[
\mathbb{P}(\mathscr{D}) =1,
\] 
where $\mathscr{D}$ was defined in Definition \ref{def:dom}.
\end{thm}
\begin{thm}\label{Thm: fixation}
Under the assumptions of Theorem \ref{Thm: dominance} we have that
\begin{equation}\label{Eq: Fixation_happens}
\Pbold(\mathscr{F})=1,
\end{equation}
where $\mathscr{F}$ was defined in Definition \ref{def:fix}.
 \end{thm}

\begin{rem}\label{Rem: fixation versus dominance}
Ideally, Theorem \ref{Thm: dominance} should be stated as a corollary of Theorem \ref{Thm: fixation}. 
However, we state the theorems in reverse order. This is because, we will prove Theorem \ref{Thm: fixation} using Theorem \ref{Thm: dominance}. 
\end{rem}

\section{Proofs}\label{sec:proofs}

The proof of Theorem \ref{Thm: dominance} will be divided into four parts. The first Lemma \ref{lem:eqpoints} will identify equilibrium points for the mean-field function $h$. The \textit{equilibrium set} of a function $h$ is the set $$\mathcal{E}(h) = \{ y\in [0,1]^d: h(y) =0\}.$$ We call \textit{trivial equilibrium points} the standard orthonormal basis of $\mathbb{R}^d$, $\{e_1,...,e_d\}$. Lemma \ref{lem:conv} will ensure that the process $(\theta_n)_{n\geq 0}$ converges and the possible limit points are given by the equilibrium points of the mean-field function $h$. Proposition \ref{Lemma: e_unstable_almost_surely} will show that almost surely the process will infinitely often move away by a small enough distance from the non-trivial equilibrium points. 
Finally Lemma \ref{lem:contrivial} will conclude that the only possible limit points are the trivial equilibrium points.

\begin{lem}\label{lem:eqpoints}
Let $f \in \mathcal{R}$. Then the equilibrium set is equal to
\be \label{Eq: Equilibrium_points}
\mathcal{E}(h) =\left \{ y \in V: y_j = \frac{1}{d-|I|} \text{ for } j\in I^c \text{ and } y_i=0 \text{ for } i \in I \right\}, 
\ee
where 
$I=\{i\in \{1,...,d\}: y_i=0 \}$.

\end{lem}

\begin{proof}
To find the equilibrium points of $h$, we need to find all solutions $y \in \{y \in \Rbold^d_+ \setminus\{0\}: \sum_{i=1}^d y_i=1\}$ for the set of self-consistency equations given by 

\begin{equation}
\label{Eq: self_similar}
\begin{split}
y_1  = \frac{f(y_1)}{\| f_d(y)\|_1}, \cdots, \,  y_d  = \frac{f(y_d)}{\| f_d(y)\|_1}.
\end{split}
\end{equation}

Since $y \in \{y \in \Rbold^d_+ \setminus\{0\}: \sum_{i=1}^d y_i=1\}$, it follows immediately that there exists some $i$, such that $y_i>0$, and thus $\| f_d(y)\|_1 >0 $. It is  clear that the unit coordinate vectors $e_i$ for $1 \le i \le d$ are equilibrium points by assumption {\bf(A)}. 

So let us assume that $y_1,..., y_k \neq 0$ for some $2 \le k \le d$. Observe that to prove \eqref{Eq: Equilibrium_points}, it is enough to show $y_1=y_2=\ldots= y_k= \frac{1}{k}$. The previous set of equations \eqref{Eq: self_similar} can be written as
\begin{equation}\label{Eq: deducing_equilibrium_points}
\begin{split}
\| f_d(y)\|_1  = \frac{f(y_1)}{y_1}, \cdots, \, \| f_d(y)\|_1 = \frac{f(y_k)}{y_k}, 
\end{split}
\end{equation}
and $f(y_i)=0$ for all $(k+1) \le i \le d$ by assumption {\bf(A)}. W.l.o.g., we may assume that $y_1 > y_2$. Then by Lemma \ref{Lemma: properties_of_f} (iii), we know that $\frac{f(y_1)}{y_1} > \frac{f(y_2)}{y_2}$, which is a contradiction to \eqref{Eq: deducing_equilibrium_points}. This shows that $y_1=y_2$, and in particular, $y_1=y_2=\ldots=y_k$ and also $y_1=y_2=\ldots=y_k=\frac{1}{k}$, since 
$\sum_{i=1}^k y_i =1$.

\end{proof}

Note that the statement of  Lemma \ref{lem:eqpoints} is similar to Proposition 2.4 from \cite{LauPa}. They prove the statement under the assumption that $f$ is concave or convex, which we do not need here. The following lemma is a consequence of Corollary \ref{cor:conveqpoints} from the appendix and Lemma \ref{lem:eqpoints}.

\begin{lem}\label{lem:conv}
Let $(\theta_n)_{n\geq 0}$ be a stochastic process defined by the recursion
\[
\theta_{n+1} = \theta_n+\frac{\sigma_{n+1}}{\tau_{n+1}}\left[ h(\theta_n)+ \frac{1}{\sigma_{n+1}}\Delta M_{n+1} \right].
\]
Then 
\[
\theta_n \overset{a.s.}\longrightarrow \theta^* 
\]
as $n\rightarrow \infty$, where $\theta^* \in \mathcal{E}(h)$.
\end{lem}
The following proposition is actually a  stronger version of Proposition 4.1 of \cite{Si-2018}, where we show that for any $ 1< k \le d$,  $e=\frac{1}{k}(e_1+e_2+\ldots+e_k)$ is likely to be unstable. The case $d=2$ is covered in Proposition 4.1 of \cite{Si-2018}, so we extend this to $d \ge 3$.

\begin{prop}\label{Lemma: e_unstable_almost_surely}
Let $\left(\sigma_n \right)_n$ and $\left(\tau_n\right)_n$ be as in the assumptions of Theorem \ref{Thm: dominance}. For any $e=\frac{1}{k}(e_1+e_2+\ldots+e_k)$, where $1 < k \le d$, there exists $\delta_n>0$, such that
\begin{itemize}
\item[(i)]
\be
\displaystyle \lim_{n \to \infty} \delta_n =0,
\ee 
\item[(ii)] 
\begin{equation}\label{Eq: choice_of_delta}
\displaystyle\sum_{n \ge 1}\frac{\delta_n\sigma_{n+1}}{\tau_{n+1}}< \infty.
\end{equation}
\item[(iii)]
For any such sequence $\left(\delta_n\right)_{n \ge 0}$
\begin{equation}\label{Eq: e_unstable_almost_surely} 
\Pbold \left((\theta_{n,1}, \theta_{n,2}\ldots, \theta_{n,d-1},\theta_{n,d} )\in B_{d}^c (e, \delta_n) \text{ i.o.} \right)=1,
\end{equation}
where for $a \in \Rbold^{d}$ and $r>0$, $B_{d}(a,r):=\{x \in \Rbold^{d} : \|x -a\|_2 \le r\}$.  
\end{itemize}
\end{prop}
\begin{rem}
By symmetry, it is clear from the above Proposition \ref{Lemma: e_unstable_almost_surely} that any $e\neq e_i$, for $1 \le i \le d$ is likely to be unstable as an equilibrium point and therefore, the only possibility is that $\theta_n \longrightarrow e_i$ almost surely as $n \to \infty$, for some $1 \le i \le d$.
\end{rem}

\begin{proof}
(i)+(ii): Recall that by our choice, we have chosen $(\sigma_n)_n$ and $(\tau_n)_{n}$, such that, $\sum_{n}\frac{\sigma_n}{\tau_n} =\infty$ and $\sum_{n}\frac{\sigma^2_n}{\tau^2_n}< \infty$, therefore there always exists $\delta_n$ that satisfies $\sum_{n \ge 1}\frac{\delta_n\sigma_{n+1}}{\tau_{n}+\frac{1}{2}\sigma_{n+1}}< \infty$, which implies that $\sum_{n \ge 1}\frac{\delta_n\sigma_{n+1}}{\tau_{n+1}}< \infty$ as $\tau_{n+1}=\tau_n+\sigma_{n+1}$. We can simply choose $0 \le \delta_n \le \frac{\sigma_{n+1}}{\tau_{n}+2 \sigma_{n+1}}$. It is immediately clear that $\lim_{n\to \infty}\delta_n=0$.\\
(iii):
For $e=\frac{1}{k}(e_1+e_2+\ldots+e_k)$, where $1 < k \le d$, define 
\[
\HH_{m}(e):=\{(\theta_{n,1}, \theta_{n,2}\ldots, \theta_{n,d-1}, \theta_{n,d} )\in B_{d}(e, \delta_n) \text{ for all } n \ge m\}.
\]

Since $\HH_{m}(e)$ is non-increasing, therefore it is enough to show that $$\lim_{m \to \infty}\Pbold(\HH_{m}(e))=0$$ to show \eqref{Eq: e_unstable_almost_surely}.

Recall that from \eqref{Eq:Fundamental_equation}, we have 
\ben
\theta_{n+1,1} = \frac{\tau_n \theta_{n,1} + B_{n+1,1}}{\tau_{n+1}},
\een
 where given $\mathcal{F}_n$, $B_{n+1,1}$ follows $Bin(\sigma_{n+1}, \Psi(\theta_{n,1}))$, $\Psi$ is short for
\begin{equation}\label{def:Psi}
\Psi(x_i) := \frac{f(x_i)}{\| f_d(x)\|_1},
\end{equation}
and $x=(x_1,\cdots, x_d) \in \mathbb{R}_+^d$, $f\in \mathcal{R}$.
 Re-write the above equation as 
\be\label{Eq: theta_exp_normal}
\theta_{n+1,1} = \frac{\tau_n \theta_{n,1} + \sigma_{n+1}\Psi(\theta_{n,1})+\eps_n \sqrt{\sigma_{n+1}\Psi(\theta_{n,1})(1-\Psi(\theta_{n,1}))}}{\tau_{n+1}},
\ee
where $\eps_n= \frac{B_{n+1,1}-\sigma_{n+1}\Psi(\theta_{n,1})}{\sqrt{\sigma_{n+1}\Psi(\theta_{n,1})(1-\Psi(\theta_{n,1}))}}.$ 
Let $\Psi_d:\mathbb{R}^d \rightarrow \mathbb{R}^d$ by $y\mapsto (\Psi(y_1),\cdots, \Psi(y_d))$.
By Taylor's theorem, we have for any $x \in \{y \in \Rbold^d_+ \setminus\{0\}: \sum_{i=1}^d y_i=1\}$
\be\label{Eq: Taylor_expansion_psi}
\Psi_d(x)- e=\sum_{i=1}^{d-1} \frac{\partial \Psi_d(\xi_x)}{\partial x_i}(x_i-\tilde{e}_i),
\ee where $\xi_x$ lies on the straight line joining $x$ and $e=(\tilde{e}_1, \tilde{e}_2\ldots, \tilde{e}_d) \in \mathcal{E}(h)$, since $\Psi(\theta_{n,1})=\tilde{e}_1$. 
Re-writing \eqref{Eq: theta_exp_normal}, using the Taylor expansion, we get 
\begin{align}
\nonumber \theta_{n+1,1}-\tilde{e}_1 & = & \frac{\tau_n}{\tau_{n+1}}\left(\theta_{n,1}-\tilde{e}_1 \right)+ \frac{\sigma_{n+1}}{\tau_{n+1}}\sum_{i=1}^{d-1} \frac{\partial \Psi_d(\xi_{\theta_n})}{\partial x_i}(\theta_{n,i}-\tilde{e}_i)\\
 & & \quad \quad + \frac{1}{\tau_{n+1}}\eps_n \sqrt{\sigma_{n+1}\Psi(\theta_{n,1})(1-\Psi(\theta_{n,1}))}.
\end{align}

\be\label{Eq: theta_Taylor_normal}
\theta_{n+1,1}-\tilde{e}_1=\k_n(\theta_n) \left(\theta_{n,1}-\tilde{e}_1 \right)+ Q_n(\theta_n)+\frac{1}{\tau_{n+1}}\eps_n \sqrt{\sigma_{n+1}\Psi(\theta_{n,1})(1-\Psi(\theta_{n,1}))}, 
\ee
where $\k_n(\theta_n)= \left(\frac{\tau_n}{\tau_{n+1}}+\frac{\sigma_{n+1}}{\tau_{n+1}}\frac{\partial \Psi_d(\xi_{\theta_n})}{\partial x_1}\right)$, and 
\begin{equation}\label{def:Q_n}
Q_n(\theta_n)= \frac{\sigma_{n+1}}{\tau_{n+1}}\sum_{i=2}^{d-1} \frac{\partial \Psi_d(\xi_{\theta_n})}{\partial x_i}(\theta_{n,i}-\tilde{e}_i).
\end{equation}

Re-iterating the above equation, we obtain
\begin{eqnarray}\label{Eq: theta_Taylor_normal_upto_m}
\nonumber \theta_{n,1}-\tilde{e}_1 & = &\prod_{j=m}^{n-1}\k_j(\theta_j) \left[\theta_{m,1}-\tilde{e}_1 +\sum_{j=m}^{n-1}\frac{1}{\left(\prod_{l=m}^{j}k_l(\theta_l)\right)} Q_j(\theta_j)\right] \\
\nonumber & & \quad +\prod_{j=m}^{n-1}\k_j(\theta_j) \sum_{j=m}^{n-1}\frac{1}{\left(\prod_{l=m}^{j}k_l(\theta_l)\right)}\eps_j \frac{\sqrt{\sigma_{j+1}\Psi(\theta_{j,1})(1-\Psi(\theta_{j,1}))}}{\tau_{j+1}}, 
\end{eqnarray}

By continuity of the partial derivatives of $\Psi_d$, for $x \in B_{d-1}(e, \delta_n)$ for all $n$ (large enough), we have $\frac{\partial \Psi_d(\xi_x)}{\partial x_i} = \frac{\partial \Psi_d(e)}{\partial x_i}+\OO(\delta_n)$.  

\noindent\textbf{Step I}:  In this step we analyse the asymptotic behavior of $\prod_{j=1}^n k_l(\theta_j)$. 
Observe that at an equilibrium point $e=\left(\tilde{e}_1, \ldots, \tilde{e}_d\right)$
\[ \frac{\partial \Psi_d(e)}{\partial x_1}= \frac{f'(\tilde{e}_1)}{f(\tilde{e}_1)+\ldots+f(\tilde{e}_d)}-\frac{f(\tilde{e}_1)\left(f'(\tilde{e}_1)-f'(1-\tilde{e}_1-\tilde{e}_2-\ldots-\tilde{e}_{d-1})\right)}{\left(f(\tilde{e}_1)+\ldots+f(\tilde{e}_d)\right)^2}.
\]
Therefore,  
\be\label{Eq: partial_derivative_first_coordinate} 
\frac{\partial \Psi_d(e)}{\partial x_1}=
\begin{cases} \alpha (\tilde{e}_1) , &  \text{ if } \tilde{e}_1 = \tilde{e}_d, \\               
               \a(\tilde{e}_1)-\frac{f(\tilde{e}_1)\left(f'(\tilde{e}_1)-f'(0)\right)}{\left(f(\tilde{e}_1)+\ldots+f(\tilde{e}_d)\right)^2}  , & \text{ if } \tilde{e}_d =0. 
\end{cases} 
\ee
From Assumption {\bf(C)}, \eqref{Eq: deducing_equilibrium_points} and the observation that $f'(0) \ge 0$, we have 
\be\label{Eq: partial_derivative_first_coordinate_bound} 
\frac{\partial \Psi_d(e)}{\partial x_1}
\begin{cases} = \alpha (\tilde{e}_1) >1 , &  \text{ if } \tilde{e}_1 = \tilde{e}_d, \\               
               \ge \left(1-\tilde{e}_1\right)\a(\tilde{e}_1), & \text{ if } \tilde{e}_d =0,
\end{cases} 
\ee
where $\alpha(x) := \frac{x f'(x)}{f(x)}$ and $f\in \mathcal{R}$.
Therefore, on $\HH_m$ we obtain for $e=\frac{1}{k}(e_1+e_2+\ldots + e_k)$
\[
\label{Eq: kappa_replaced_derivatives}
\nonumber \k_n(\theta_n)=\begin{cases}
\frac{\tau_n}{\tau_{n+1}}+\frac{\sigma_{n+1}}{\tau_{n+1}}\left(\frac{k-1}{k}\a(\tilde{e}_1) +\OO(\delta_{n})\right),  &  \text{ if } 1< k \le (d-1), \\
\frac{\tau_n}{\tau_{n+1}}+\frac{\sigma_{n+1}}{\tau_{n+1}}\left(\a(\tilde{e}_1) +\OO(\delta_{n})\right),  &  \text{ if } k=d.
\end{cases}
\]
Writing $\pi_{m,n-1}(k):= \prod_{j=m}^{n-1} \left(\frac{\tau_j+ \beta_k \sigma_{j+1}}{\tau_{j+1}}\right)$
\ben
\prod_{j=m}^{n-1}\k_j(\theta_j)=\pi_{m,n-1}(k)\prod_{j=m}^{n-1}\left(1+\frac{\OO(\delta_{j})\sigma_{j+1}}{\tau_j+\beta_k \sigma_{j+1}}\right),
\een where $\beta_k = \frac{k-1}{k}\a(\tilde{e}_1)$, if $1 < k \le d-1$, and $\beta_k = \a(\tilde{e}_1)$ for $k=d$. 
\[
\prod_{j=m}^{n-1}\left(1+\frac{\OO(\delta_{j})\sigma_{j+1}}{\tau_j+\beta_k \sigma_{j+1}}\right)=\exp\left(\sum_{j=m}^{n-1}\frac{\OO(\delta_{j})\sigma_{j+1}}{\tau_j+\beta_k \sigma_{j+1}}\right).
\]
Since $1 \ge (1-\frac{1}{k}) \ge \frac{1}{2}$ and $\a(\tilde{e}_1) >1$,  
by our choice of $\delta_j$, in \eqref{Eq: choice_of_delta}, $\sum_{j}\frac{\delta_j}{\tau_j+\beta_k \sigma_{j+1}}< \infty$. 
Hence, for all $m$ large enough 
\[
\prod_{j=m}^{n-1}\left(1+\frac{\OO(\delta_{j})\sigma_{j+1}}{\tau_j+\beta_k \sigma_{j+1}}\right)=(1+o(1)).
\]
Therefore, for all $m$ large enough and on $\HH_m$, we have  
\be\label{Eq: kappa_small_oh}
\prod_{j=m}^{n-1}\k_j(\theta_j)=\pi_{m,n-1}(k)(1+o(1)).
\ee 
\noindent\textbf{Step II}: In this step we analyse the sum over $Q_j(\theta_j)$, where $Q_j$ was defined in \eqref{def:Q_n}.

From our Assumption {\bf(B)}, for $2 \le i \le k$, we have $\frac{\partial \Psi_d(\xi_x)}{\partial x_i} = \frac{\partial \Psi_d(e)}{\partial x_i}+\OO(\delta_n)$ is bounded for all $n \ge m$.

On $\HH_m$, 
\be\label{Eq: simplified_Qj}
Q_j(\theta_j) =\OO(\delta_j)\frac{\sigma_{j+1}}{\tau_{j+1}}.
\ee

Recall that in {\bf Step I} we proved that $\prod_{j=m}^{l}\k_j(\theta_j)=\pi_{m,l}(k)(1+o(1))$, where 
$\pi_{m,l}(k)=\prod_{j=m}^l \left(\frac{\tau_j+ \beta_k \sigma_{j+1}}{\tau_{j+1}}\right)$.

\noindent \textbf{Case i}: When $k=d$, 
$$\pi_{m,l}(k)=\exp \left(\sum_{j=m+1}^l \log\left(1+(\a(\tilde{e}_1)-1)\frac{\sigma_{j+1}}{\tau_{j+1}}\right)\right) \longrightarrow \infty,$$ as $l\rightarrow \infty$ since $\a(\tilde{e}_1) >1$ and $\sum_{j} \frac{\sigma_j}{\tau_j} = \infty$.
Therefore, 
\be \label{Eq: sum_over_Q_k_d}
\nonumber \sum_{j=m}^{n-1}\frac{1}{\left(\prod_{l=m}^{j}k_l(\theta_l)\right)} Q_j(\theta_j)= (1+o(1))\sum_{j=m}^{n-1}
\exp \left(- \sum_{j=m+1}^l (\a(\tilde{e}_1)-1)\frac{\sigma_{j+1}}{\tau_{j+1}}\right) \frac{\sigma_{j+1}}{\tau_{j+1}}\OO(\delta_j) < \infty,
\ee by our choice of $\delta_j$.

\noindent \textbf{Case ii}: When $1 <k \le d-1$. 
$\pi_{m,l}(k)=\exp \left(\sum_{j=m+1}^l \log\left(1+(\frac{k-1}{k}\a(\tilde{e}_1)-1)\frac{\sigma_{j+1}}{\tau_{j+1}}\right)\right)=\exp \left(\sum_{j=m+1}^l (\frac{k-1}{k}\a(\tilde{e}_1)-1)\frac{\sigma_{j+1}}{\tau_{j+1}}\right) + o(1)$ by the assumptions of Theorem \ref{Thm: dominance}. 

If $\a(\tilde{e}_1)\ge  \frac{k}{k-1}$, then it follows similar to Case i, that 
\be \label{Eq: sum_over_Q_k_alpha_greater_2}
\begin{split}
& \nonumber \sum_{j=m}^{n-1}\frac{1}{\left(\prod_{l=m}^{j}k_l(\theta_l)\right)} Q_j(\theta_j)\\
& = (1+o(1))\sum_{j=m}^{n-1}
\exp \left(- \sum_{j=m+1}^l \left(\frac{k-1}{k}\a(\tilde{e}_1)-1\right)\frac{\sigma_{j+1}}{\tau_{j+1}}\right) \frac{\sigma_{j+1}}{\tau_{j+1}}\OO(\delta_j) < \infty.
\end{split}
\ee
If $ \a(\tilde{e}_1) < \frac{k}{k-1}$, then because of our choice of $\delta_j$ for all $m,n$ (large enough)
\begin{equation}\label{Eq: sum_over_Q_k_alpha_lessthan_2}
\begin{split}
&  \left(\prod_{j=m}^{n-1}\k_j(\theta_j)\right)\sum_{j=m}^{n-1}\frac{1}{\left(\prod_{l=m}^{j}k_l(\theta_l)\right)} Q_j(\theta_j)  \\
= & \sum_{j=m}^{n-1} \left(\prod_{l=j+1}^{n-1}\k_l(\theta_l)\right) Q_j(\theta_j)\\
 & =  (1+o(1)) \sum_{j=m}^{n-1} \exp \left(\sum_{l=j+1}^{n-1} \left(\frac{k-1}{k}\a(\tilde{e}_1) -1\right)\frac{\sigma_{l+1}}{\tau_{l+1}}\right)\OO(\delta_{j})\frac{\sigma_{j+1}}{\tau_{j+1}} < \epsilon ,
\end{split}
\end{equation}
 since $\frac{k-1}{k} \a(\tilde{e}_1) <1$, $\sum_j \frac{\sigma_{j+1}}{\tau_{j+1}} = \infty$ and $\sum_{j} \delta_j \frac{\sigma_{j+1}}{\tau_{j+1}}< \infty$. 

\noindent{{\bf Step 3}}: In this step we will explore the last part of the summand in \\
$\sum_{j=m}^{n-1}\frac{1}{\left(\prod_{l=m}^{j}k_l(\theta_l)\right)}\eps_j \frac{\sqrt{\sigma_{j+1}\Psi(\theta_{j,1})(1-\Psi(\theta_{j,1}))}}{\tau_{j+1}}$. 
The argument is similar to that of Proposition 4.1 in \cite{Si-2018}, so we present only those details that are crucial and slightly different from those in \cite{Si-2018}. From Step 1, and $\Psi(\theta_{j,1})= \frac{1}{k}+\OO(\delta_j)$, we get for appropriate constants $c_k$
\[
\frac{1}{\left(\prod_{l=m}^{j}k_l(\theta_l)\right)}\eps_j \frac{\sqrt{\sigma_{j+1}\Psi(\theta_{j,1})(1-\Psi(\theta_{j,1}))}}{\tau_{j+1}}=(c_k+o(1))\eps_j \frac{\sqrt{\sigma_{j+1}}}{\pi_{m,j}\tau_{j+1}}.
\]

We know that 
if $X \sim Bin (n,p)$, then

$\Ebold \left[\exp \left( it \frac{X-np}{\sqrt{np(1-p)}}\right)\right]= \left(1-\frac{t^2}{2n}+\OO\left(\frac{t^3}{n^{\frac{3}{2}}}\right)\right)^n = \exp\left(-\frac{t^2}{2}+ \OO \left(\frac{t^3}{\sqrt{n}} \right)\right)$, when $|t| < \delta$ for some $\delta>0$. Therefore, 
\ben
\Ebold \left[\exp \left(it \eps_j \right)\mid \mathcal{F}_{j-1}\right]= \exp\left(-\frac{t^2}{2}+ \OO \left(\frac{t^3}{\sqrt{\sigma_{j+1}}} \right)\right),
\een 
since given $\mathcal{F}_{j-1}$, $\eps_j \sim Bin (\sigma_j, \Psi(\theta_{j-1,1}))$.
And hence, we have by tower-property
\ben 
\Ebold \left[\exp \left(it \sum_{j=m}^{n-1}\eps_j \right)\biggl | \mathcal{F}_{m-1}\right]= \exp\left(\sum_{j=m}^{n-1}\frac{-t^2}{2}+ \OO \left(\frac{t^3}{\sqrt{\sigma_{j+1}}} \right)\right). 
\een
Hence,  we can write 
\be \label{Eq: Characteristic_function_conditional}
\Ebold \left[\exp \left(\frac{it}{\mu_{m,n}} \sum_{j=m}^{n-1}\eps_j \right) \biggl | \mathcal{F}_{m-1}\right]=\exp\left(-\frac{t^2}{2}+ \OO \left(\frac{t^3}{\mu^3_{m,n}} \right)\sum_{j=m}^{n-1}\frac{\sigma_{j+1}}{\pi_{m,j}^3 \tau_{j+1}^3}\right),
\ee
where $\mu_{m,n}:= \sum_{j=m}^{n-1}\frac{\sigma_{j+1}}{\pi_{m,j}^2 \tau_{j+1}^2}$.

 For $e =\frac{1}{d}(e_1+e_2+\ldots+e_d)$, it follows that $\frac{1}{\mu^3_{m,n}}\sum_{j=m}^{n-1}\frac{\sigma_{j+1}}{\pi_{m,j}^3 \tau_{j+1}^3} \longrightarrow 0$ as $m,n \to \infty$. 

When $e =\frac{1}{k}(e_1+e_2+\ldots+e_k)$ for $1 <k \le d-1$, and $\sup_{n} \sigma_n =\infty$, exactly same argument as in \cite{Si-2018} gives us $\frac{1}{\mu^3_{m,n}}\sum_{j=m}^{n-1}\frac{\sigma_{j+1}}{\pi_{m,j}^3 \tau_{j+1}^3} \longrightarrow 0$ as $m,n \to \infty$. 

For $e=\frac{1}{k}(e_1+e_2+\ldots+e_k)$ for $1 <k \le d-1$, and $\sup_{n} \sigma_n < \infty$, we observe that if $\a(\tilde{e}_1)> \frac{k}{k-1}$ then similar argument as in \cite{Si-2018} works. 

Therefore, the only case we need to discuss in details is for $e=\frac{1}{k}(e_1+e_2+\ldots+e_k)$ for $1 <k \le d-1$, and $\sup_{n} \sigma_n < \infty$, and $\a \le \frac{k}{k-1}$.

It is easy to see that 

\begin{eqnarray}\label{Eq: kappa_alpha_le_k-1/k}
 \prod_{j=m}^{n-1}\k_j(\theta_j)= \exp \left(- \sum_{j=m+1}^{n-1} \left(\frac{k-1}{k}\a-1+\OO(\delta_j) \right)\frac{\sigma_{j+1}}{\tau_{j+1}}\right)\longrightarrow 0, \text{ as } m,n \to \infty, 
\end{eqnarray}
since $\lim_{n \to \infty} \delta_n =0$, and $\sum_{j}\frac{\sigma_j}{\tau_j} =\infty$.

\begin{eqnarray} \label{Eq: Characteristic_function_conditional_alpha_le_k-1/k}
\nonumber & & \Ebold \left[\exp \left( \frac{it}{\Gamma_{m ,n-1}} \sum_{j=m}^{n-1}\left(\prod_{l=j+1}^{n-1}k_l(\theta_l)\right)\eps_j \frac{\sqrt{\sigma_{j+1}}}{\tau_{j+1}}\right)\biggl | \mathcal{F}_{m-1}\right]\\
& = & \exp\left(-\frac{t^2}{2} + \OO(1)\frac{t^3}{\Gamma^3_{m ,n-1}}\sum_{j=m}^{n-1}\left(\prod_{l=j+1}^{n}\k_l(\theta_l)\right)^3\frac{\sigma_{j+1}}{\tau_{j+1}^3}\right),
\end{eqnarray}
where $\Gamma_{m ,n-1}:=\left[\sum_{j=m}^{n-1}\left(\prod_{l=j+1}^{n-1}\k_l(\theta_l)\right)^2\frac{\sigma_{j
+1}}{\tau^2_{j+1}}\right]^{\frac{1}{2}}$.  

Since $\sup_n \sigma_n <\infty$, $ j \le \tau_j \le C j $ for suitable constant $C>0$. From \eqref{Eq: kappa_alpha_le_k-1/k}, we get 
 
\be
\sum_{j=m}^{n-1}\left(\prod_{l=j+1}^{n-1}\k_l(\theta_l)\right)^2 \frac{\sigma_{j
+1}}{\tau^2_{j+1}} = \OO(1) \sum_{j=m}^{n-1} \frac{1}{j^2} = \OO(1) \left(\frac{1}{m}-\frac{1}{n}\right)
\ee 

By similar arguments, 
\be
\sum_{j=m}^{n-1}\left(\prod_{l=j+1}^{n-1}\k_l(\theta_l)\right)^3 \frac{\sigma_{j
+1}}{\tau^3_{j+1}} = \OO(1) \sum_{j=m}^{n-1} \frac{1}{j^3}= \OO(1) \left(\frac{1}{m^2}-\frac{1}{n^2}\right)
\ee

Therefore, as $m \to \infty$,
\be
\frac{1}{\Gamma^3_{m ,n-1}}\sum_{j=m}^{n-1}\left(\prod_{l=j+1}^{n}\k_l(\theta_l)\right)^3\frac{\sigma_{j+1}}{\tau_{j+1}^3}= \OO(1)\frac{1}{\sqrt{m}}\longrightarrow 0.
\ee

For this case, choose the subsequence $n_m=2m$ and choose $\delta_n= \OO(1) \frac{1}{n^{\frac{1}{2}+\delta}}$, for some $\delta>0$. Finally from all previous observations,
\be \label{Eq: dominance_by_Normal}
\Pbold(\HH_{m}(e)) \le \Pbold \left(\Gamma_{m ,n_m} N \in \left[-\delta_{n_m}, \delta_{n_m}\right]\right)\longrightarrow 0 \text{ as } m \to \infty,
\ee
since $\frac{\delta_{n_m}}{\Gamma_{m ,n_m}}=\OO(1) \frac{1}{m^{\delta}} \longrightarrow 0$ as $m \to \infty$. 
\end{proof}

\begin{prop}\label{lem:contrivial}
Let $(\theta_n)_{n\geq 0}$ be the process defined recursively in \eqref{Eq:fundamental_recursion_mean_field}. Then
\[
\theta_n \overset{a.s.}\longrightarrow e^*,
\]
as $n\rightarrow \infty$ where $e^* \in \{e_1,...,e_d \}$ are the trivial equilibrium points.
\end{prop}
The proof of this proposition is similar to parts of Theorem 1.2. from \cite{Si-2018} adapted appropriately for the class of reinforcement functions $\mathcal{R}$, the main difference is in the final argument leading upto the conclusion. We present all details for the sake of completeness. 

\begin{proof}
For each $n\in \mathbb{N}$ and $i \in \{1,..., d\}$ we have that

\begin{equation}
\theta_{n+1, i}  = \frac{\tau_n \theta_{n,i} + B_{n+1, i}}{\tau_{n+1}}\\
\end{equation}
where $B_{n,i} \sim Bin \left(\sigma_n, \Psi(\theta_{n, i}) \right)$ given $\mathcal{F}_n$, where $\Psi$ was defined in \eqref{def:Psi}. We can express the system of equations for $i\in \{1,...,d\}$ by 
\begin{equation}
\begin{split}
\theta_{n+1, i} &= \theta_{n,i} + \frac{\sigma_{n+1}}{\tau_{n+1}} \left ((\Psi(\theta_{n,i})- \theta_{n,i}) +\frac{1}{\sigma_{n+1}} (B_{n+1} -\sigma_{n+1}\Psi(\theta_{n,i}) )  \right )\\
& = \theta_{n,i} + \frac{B_{n+1, i} -\sigma_{n+1} \Psi(\theta_{n,i})}{\tau_{n+1}} - \frac{\sigma_{n+1}}{\tau_{n+1}} (\theta_{n,i} - \Psi(\theta_{n,i})).
\end{split}
\end{equation}
For every $\eta \geq 0,$ we have by the above recursive relations that

\begin{equation}\label{eq:rec}
\theta_{n+\eta, i}  = \theta_{\eta, i} + M_{n, i}(\eta) - R_{n, i}(\eta) \\
\end{equation}
where 
\[
M_{n,i}(\eta) = \sum_{j= \eta+1}^{n+\eta} \frac{B_{j, i} -\sigma_{j} \Psi(\theta_{j-1, i})}{\tau_{j}}
\]
for $i=1,...,d$ and
\[
R_{n,i}(\eta) = \sum_{j=\eta+1}^{\eta+n} \frac{\sigma_j}{\tau_j} (\theta_{j-1, i} - \Psi(\theta_{j-1, i})).
\]
Then the vector $M_n(\eta) = (M_{n,1}(\eta),...,M_{n,d}(\eta))$ is a martingale w.r.t. $(\mathcal{F}_{n+ \eta })_n$, since $\Ebold\left[B_{n+1, i}\mid \mathcal{F}_n\right]=\sigma_{n+1} \Psi(\theta_{n, i})$. Note that up to now, we did not use the specific definition of $\eta$, which we will define later. For simplicity of notation, we write $M_{n,i}$ instead of $M_{n,i}(\eta)$.
Let us show that the martingale is bounded in $\L^2(\mathbb{R}^d)$. Note that we can write
\[
M_{n,i} = M_{n-1,i} +  \frac{B_{n+\eta,i} -\sigma_{n+\eta} \Psi(\theta_{n+\eta-1, i})}{\tau_{n+\eta}}
\]
and further estimate
\begin{equation}\label{Eq:L_2_boundedness_martingale_shown}
\begin{split}
\mathbb{E} (\langle M_{n}, M_{n} \rangle | \mathcal{F}_{\eta}) & = \sum_{i=1}^d  \mathbb{E} ( (M_{n,i})^2 | \mathcal{F}_{\eta}) \\
& =  \sum_{i=1}^d  \mathbb{E} ( (M_{n-1,i})^2 | \mathcal{F}_{\eta}) +    \mathbb{E} \left ( \left(  \frac{B_{n+\eta,i} -\sigma_{n+\eta} \Psi(\theta_{n+\eta-1, i})}{\tau_{n+\eta}}  \right)^2 \biggl| \mathcal{F}_{\eta} \right) \\
& + 2\sum_{i=1}^d \mathbb{E} \left ( \left [ \frac{B_{n+\eta, i} -\sigma_{n+\eta} \Psi(\theta_{n+\eta-1, i})}{\tau_{n+\eta}} \right ] M_{n-1} \biggl| \mathcal{F}_{\eta} \right)  \\
& = \sum_{i=1}^d  \mathbb{E} ( (M_{n-1,i})^2 | \mathcal{F}_{\eta}) +    \frac{1}{\tau^2_{n+\eta}}\mathbb{E} \left (  \sigma_{n+\eta} \Psi(\theta_{n+\eta-1, i} )(1-\Psi(\theta_{n+\eta-1, i}))  \biggl| \mathcal{F}_{\eta} \right) 
\end{split}
\end{equation}

Since $\sigma_{n+\eta} \Psi(\theta_{n+\eta-1, i} )(1-\Psi(\theta_{n+\eta-1, i}) \leq \sigma_{n+\eta}$ and by recursivity we can estimate
\[
\mathbb{E} (\langle M_{n}, M_{n} \rangle | \mathcal{F}_{\eta}) \leq d \sum_{j=\eta+1}^{\eta+n} \frac{\sigma_j}{\tau^2_j}.
\]

\begin{equation} \label{Eq: sum_over_inverse_tau_to_integral}
\begin{split}
\sum_{j=\eta+1}^{\eta+n} \frac{\sigma_j}{\tau^2_j} & = \sum_{j=\eta+1}^{\eta+n} \int_{\tau_{j-1}}^{\tau_j}{\frac{1}{\tau^2_j} \mathrm{d}x} \le \sum_{j=\eta+1}^{\infty} \int_{\tau_{j-1}}^{\tau_j}{\frac{1}{x^2} \mathrm{d}x} = \int_{\tau_{\eta}}^{\infty} {\frac{1}{x^2} \mathrm{d}x} = \frac{1}{\tau_{\eta}},
\end{split}
\end{equation}

so that $M_n$ is bounded in $L^2(\mathbb{R}^d)$:
\be\label{Eq: L2_bounded_martingale}
\mathbb{E} (\mathbb{E} (\langle M_{n}, M_{n} \rangle | \mathcal{F}_{\eta}) )= \mathbb{E} (\langle M_{n}, M_{n} \rangle)  \leq \frac{d}{\tau_{\eta}}
\ee
which implies that $M_n$ converges almost surely and in $\L^2(\mathbb{R}^d)$  to some $M_{\infty} \in \L^2(\mathbb{R}^d)$. 

Observe that there exists a sequence $\left(\delta_n\right)_{n \ge 0}$ (see Lemma 5.1 from \cite{Si-2018}), such that, it satisfies (i) and (ii) from Proposition \ref{Lemma: e_unstable_almost_surely}, and 
\ben
\lim_{n \to \infty} \frac{1}{\delta_n^2 \tau_n}=0.
\een

For any $\eps>0$, choose $N_1>0$, such that, for all $n \ge N_1$
\be\label{Eq: prob_conv_martingale_zero}
\Pbold\left(M_{n,1}(N_1)\ge \frac{\delta_{N_1}}{2}\right) \le \frac{4 d}{\delta_{N_1}^2 \tau_{N_1} }\le \eps.
\ee

By Proposition \ref{Lemma: e_unstable_almost_surely}, for $\delta_n $ as in \eqref{Eq: prob_conv_martingale_zero} W.l.o.g., we may assume that $\Pbold (N_2<\infty)>0$, where $N_2:= \inf \left\{n: \theta_{n,1}<\frac{1}{k}- \delta_n\right\}$. Choose $\eta = \max \{N_1,N_2\}$. Henceforth, we will work with $M_{n,1}(\eta)$ and $R_{n,1}(\eta)$ for this choice of $\eta$. As before, we write $M_{n,1}$ and $R_{n,1}$ for notational simplicity. 

Lemma \ref{lem:conv} yields that $\theta_n \rightarrow \theta^*$ as $n\rightarrow \infty$, hence together with $M_n \rightarrow M_{\infty}$ a.s.,  
we deduce that $R_n$ has to converge as well. It remains to show that $\theta^* \neq e$ where $e=\frac{1}{k}(e_1+...+e_k)$ and $2 \leq k \leq d$. 
Assume that $\theta^*=e$. Recall that $\Psi(\theta^*_{i})=\frac{1}{k}$ for $1 \le i \le k$, and $\Psi(\theta^*_{i})=0$ for all other $i$. Since $\Psi\left(\cdot\right)$ is continuous, and $\mathcal{E}(h)$ consists of isolated points, either $\Psi\left(x\right)<x$ for all $x \in (\frac{1}{k+1}, \frac{1}{k})$ and $x<\Psi\left(x\right)$ for all $x \in (\frac{1}{k}, \frac{1}{k-1})$ or $\Psi\left(x\right)>x$ for all $x \in (\frac{1}{k+1}, \frac{1}{k})$ and $x>\Psi\left(x\right)$ for all $x \in (\frac{1}{k}, \frac{1}{k-1})$. W.l.o.g., we may assume that $\Psi\left(x\right)<x$ for all $x \in (\frac{1}{k+1}, \frac{1}{k})$. 

On the event $E_{\eta}:=\left\{\theta_{\eta,1}<\frac{1}{k}-\delta_{\eta}\right\} \cap \left\{M_{n,1}(\eta)\le \frac{\delta_{\eta}}{2}\right\}$, we have $\theta_{\eta, 1} - \Psi(\theta_{\eta, 1})>0$ and from \eqref{Eq: prob_conv_martingale_zero} $\Pbold (E_{\eta}) >0$ . Now let show by induction that on $E_{\eta}$, $\theta_{n,1}< \frac{1}{k}-\frac{\delta_{\eta}}{2}$ for $n \ge \eta$. The base step of the induction for $n=\eta$ is obvious. Let us assume that $\theta_{n-1,1}< \frac{1}{k} -\frac{\delta_{\eta}}{2}$, which implies that $R_{n,1}>0$ and $\theta_{j, 1} - \Psi(\theta_{j, 1}) >0$ for $\eta \le j \le n-1$. Therefore,   
\ben
\theta_{n,1} < \theta_{\eta,1}+ M_{n,1} \le \frac{1}{k} -\delta_{\eta} + \frac{\delta_{\eta}}{2} =\frac{1}{k} -\frac{\delta_{\eta}}{2}.
\een 
Therefore, on $E_{\eta}$, we have $\theta_{n, 1} - \Psi(\theta_{n, 1}) >0$ for all $n \ge \eta$ and $0<\lim_{n \to \infty}R_{n,1}<\infty$. But this is a contradiction, since on $E_{\eta}$ 
\be\label{Eq: R_n_negative}
\displaystyle \lim_{n \to \infty} R_{n,1}=\displaystyle \lim_{n \to \infty}(\theta_{\eta,1}+ M_{n,1}-\theta_{n,1}) \le -\frac{\delta_{\eta}}{2}.
\ee


\end{proof}
Our proof for Theorem \ref{Thm: fixation} is simpler than the proof of Proposition 7.1 from \cite{Si-2018}. This is due to the assumption $\sum_n \left(\frac{\sigma_n}{\tau_n}\right)^2 < \infty$.
\begin{proof}[Proof of Theorem \ref{Thm: fixation}]
From Theorem \ref{Thm: dominance}, we know that 
\begin{equation*}
\mathbb{P}(\mathcal{D}) =1.
\end{equation*}
By symmetry, W.l.o.g. we may assume $\theta_{n,1} \longrightarrow 0$ as $n \to \infty$. To prove \eqref{Eq: Fixation_happens}, it is enough to show that conditioned on the event $\theta_{n,1} \longrightarrow 0$ as $n \to \infty$, if 
$\sup_{n \ge 1} U_{n,1} = \infty$, then we get a contradiction.

Let us assume that $\sup_{n \ge 1} U_{n,1} = \infty$. Since $U_{n,1}$ is non-decreasing, it follows  
$ U_{n,1} \longrightarrow \infty,$ as $n \to \infty$.

W.l.o.g. we may assume that for some $m, n$(large enough), $U_{i+1,1}-U_{i,1} >0$ for $ n \le i \le m$. For some $C>0$ by assumption {\bf (C)},
\begin{equation}\label{Eq: Integrability_of_reinforcement_function}
\displaystyle \sum_{i=n}^{m} \frac{U_{i+1, 1}-U_{i,1}}{U_i^{\alpha}} \ge \int_{U_{n,1}}^{U_{m,1}} 
{\frac{1}{x^{\alpha}}\mathrm{d}x} \sim \frac{C}{U_{n,1}^{\alpha -1}}. 
\end{equation}

Recall that we know from the \eqref{Eq:Fundamental_equation}
\ben
U_{n+1,1} = U_{n,1} + B_{n+1,1},
\een
 where $B_{n+1,1}$ follows $Bin(\sigma_{n+1}, \Psi(\theta_{n,1}))$.

Let us re-write the above equation as 
\be\label{Eq: color_1_Normal_correction}
U_{n+1, 1} = U_{n,1} + \sigma_{n+1}\Psi(\theta_{n,1})+\eps_n \sqrt{\sigma_{n+1}\Psi(\theta_{n,1})(1-\Psi(\theta_{n,1}))},
\ee
where $\eps_n= \frac{B_{n+1,1}-\sigma_{n+1}\Psi(\theta_{n,1})}{\sqrt{\sigma_{n+1}\Psi(\theta_{n,1})(1-\Psi(\theta_{n,1}))}}.$

Since $x_1+x_2+\ldots+x_d=1$, and $f$ is non-decreasing, we have $\sum_{i=1}^d f(x_i) \ge f(\frac{1}{d}) =: C_d^{-1} $, which implies that 
\be \label{Eq: Bound_denominator_psi}
\Psi(x) \leq C_d f(x).
\ee

By the estimate from Lemma \ref{Lemma: properties_of_f} (iv)
applied to \eqref{Eq: color_1_Normal_correction}, we get 

\begin{eqnarray}\label{Eq: upper_bound_for_diff_U_n}
\nonumber U_{n+1, 1} & \le & U_{n,1} + C_d \sigma_{n+1}  \theta_{n,1}^{\alpha}
            + \eps_n \sqrt{\sigma_{n+1}\Psi(\theta_{n,1})(1-\Psi(\theta_{n,1}))} \\
\frac{U_{n+1, 1} - U_{n,1}}{U^{\alpha}_{n,1}} & \le & C_d \frac{\sigma_{n+1}}{\tau^{\alpha}_n}
             + \frac{1}{U^{\alpha}_{n,1}}\eps_n \sqrt{\sigma_{n+1}\Psi(\theta_{n,1})(1-\Psi(\theta_{n,1}))}.											
\end{eqnarray}
Thus,
\be\label{Eq: sum_over_upper_bound}
\displaystyle \sum_{i=n}^{\infty} \frac{U_{i+1, 1}-U_{i,1}}{U_{i,1}^{\alpha}} \le C_d \displaystyle
\sum_{i \ge n} \frac{\sigma_{i+1}}{\tau^{\alpha}_i} + \displaystyle \sum_{i \ge n} \frac{1}{U^{\alpha}_{i,1}}\eps_i \sqrt{\sigma_{i+1}\Psi(\theta_{i,1})(1-\Psi(\theta_{i,1}))}.
\ee

Combining this with \eqref{Eq: Integrability_of_reinforcement_function}, we get 

\be\label{Eq: lower_bound_on_color_1}
\frac{C}{U_{n,1}^{\alpha -1}} \le C_d \displaystyle
\sum_{i \ge n} \frac{\sigma_{i+1}}{\tau^{\alpha}_i} + \displaystyle \sum_{i \ge n} \frac{1}{U^{\alpha}_{i,1}}\eps_i \sqrt{\sigma_{i+1}\Psi(\theta_{i,1})(1-\Psi(\theta_{i,1}))}. 
\ee

We obtain using arguments similar to \eqref{Eq: sum_over_inverse_tau_to_integral}, for some suitable constant $C'>0$,

\ben
\sum_{i \ge n} \frac{\sigma_{i+1}}{\tau^{\alpha}_i}\le C' \int_{\tau_n}^{\infty} {\frac{1}{x^{\alpha}} \mathrm{d}x}= C' \frac{1}{\tau_n^{\alpha -1}}.
\een

Using the above bound in \eqref{Eq: lower_bound_on_color_1}, we obtain where $C, C'>0$ are suitable constants that may change accordingly  
\begin{eqnarray}\label{Eq: lower_bound_on_color_2}
\nonumber \frac{C}{U_{n,1}^{\alpha -1}}  & \le &  C' \frac{1}{\tau_n^{\alpha -1}} +  \displaystyle \sum_{i \ge n} \frac{1}{U^{\alpha}_{i,1}}\eps_i \sqrt{\sigma_{i+1}\Psi(\theta_{i,1})(1-\Psi(\theta_{i,1}))}\\
C  & \le & C' \theta_{n,1}^{\alpha -1}+ U_{n,1}^{\alpha -1} \displaystyle \sum_{i \ge n} \frac{1}{U^{\alpha}_{i,1}}\eps_i \sqrt{\sigma_{i+1}\Psi(\theta_{i,1})(1-\Psi(\theta_{i,1}))}.
\end{eqnarray}

Let us denote by 
\begin{equation}\label{Eq: T_n}
T_n := U_{n,1}^{\alpha -1} \displaystyle \sum_{i \ge n} \frac{1}{U^{\alpha}_{i,1}}\eps_i \sqrt{\sigma_{i+1}\Psi(\theta_{i,1})(1-\Psi(\theta_{i,1}))}.
\end{equation}

Using \eqref{Eq: Bound_denominator_psi} and $f(x) \le x^{\alpha}$, we get for some $C'>0$
\ben
T_n \le C' U_{n,1}^{\alpha -1} \displaystyle \sum_{i \ge n} \frac{1}{U^{\frac{\alpha}{2}}_{i,1}}\eps_i \sqrt{\sigma_{i+1} \frac{1}{\tau_i^{\alpha}}(1-\Psi(\theta_{i,1}))}.
\een
Since $U_{n,1}$ is non-decreasing in $n$, we have $\frac{U_{n,1}}{U_{i,1}} \le 1$. Therefore, from the above equation, we get  
\begin{eqnarray} \label{Eq: Upper_bound_1_on_T_n}
\nonumber T_n & \le & C' \frac{1}{U_{n,1}} \displaystyle \sum_{i \ge n} \eps_i  \sqrt{ \sigma_{i+1} \frac{U_{n,1}^{\alpha}}{\tau_i^{\alpha}}(1-\Psi(\theta_{i,1}))}\\
 \nonumber  & = & C' \frac{1}{\sqrt{U_{n,1}}}\displaystyle \sum_{i \ge n} \eps_i \sqrt{ 
		\frac{\sigma_{i+1}}{\tau_i} \left(\frac{U_{n,1}}{\tau_i}\right)^{\alpha-1}(1-\Psi(\theta_{i,1}))}\\
		& \le & C' \frac{1}{\sqrt{U_{n,1}}}\displaystyle \sum_{i \ge n} \eps_i \sqrt{ 
		\frac{\sigma_{i+1}}{\tau_i} \left(\frac{\tau_n}{\tau_i}\right)^{\alpha-1}},
\end{eqnarray}
where we have used $(1-\Psi(\theta_{i,1})) \le 1$, and $U_{n,1} \le \tau_n$.

It is easy to see that 
\begin{equation}\label{Eq: tau_n+1_by_sigma}
\frac{\sigma_{n+1}}{\tau_n} = \frac{\frac{\sigma_{n+1}}{\tau_{n+1}}}{1 - \frac{\sigma_{n+1}}{\tau_{n+1}}}.
\end{equation}

Observe that by our assumption $\sum_{n \ge 1}\frac{\sigma^2_n}{\tau^2_n} < \infty$, we have 
$\lim_{n \to \infty}\frac{\sigma_n}{\tau_n} =0$ and hence, from \eqref{Eq: tau_n+1_by_sigma} we have 
$\frac{\sigma_{n+1}}{\tau_n} \longrightarrow 0$, as $n \to \infty$.

Therefore, for all $n$(large enough), we obtain
\begin{eqnarray} \label{Eq: bound_on_ratio_tau_n_tau_i}
\nonumber \frac{\tau_i}{\tau_n} & = & \prod_{j=n}^{i-1}  \frac{\tau_{j+1}}{\tau_j}\\
 \nonumber                       & = & \prod_{j=n}^{i-1} \left(1+ \frac{\sigma_{j+1}}{\tau_j}\right)\\
                                 & = & \left(1+ o(1)\right)^{i-n}, \text{ for all } i \ge n.
\end{eqnarray}

Therefore, using the above inequality and $\alpha >1$, 
\be \label{Eq: summability_ratio_tau_n_tau_i}
\displaystyle \sum_{i \ge n}\left(\frac{\tau_n}{\tau_i}\right)^{\alpha-1} = \displaystyle \sum_{i \ge 1}
\left(\frac{1}{\left(1+ o(1)\right)^{i}} \right)^{\alpha -1} < \infty.
\ee

Similar to \eqref{Eq: L2_bounded_martingale}, it can be easily shown 
$\left(N_n= \sum_{i=1}^n \eps_i \sqrt{\frac{\sigma_{i+1}}{\tau_i} \left(\frac{\tau_n}{\tau_i}\right)^{\alpha-1}}\right)_{n \ge 1}$ is a $\L^2(\mathbb{R})$ bounded martingale, and we have for some $C>0$
\be
\sup_{n \ge 1} \Ebold \left[N_n^2\right] \le C \sum_{i=1}^{\infty} \left(\frac{\tau_n}{\tau_i}\right)^{\alpha-1}< \infty, 
\ee
where we have used \eqref{Eq: bound_on_ratio_tau_n_tau_i} and $\frac{\sigma_{n+1}}{\tau_n} \longrightarrow 0$, as $n \to \infty$.

Since $N_n$ is an $L^2(\mathbb{R})$ bounded, $N_n$ converges almost surely and in $L^2(\mathbb{R})$.
This implies that $\left(\sum_{i \ge n } \eps_i \sqrt{\frac{\sigma_{i+1}}{\tau_i} \left(\frac{\tau_n}{\tau_i}\right)^{\alpha-1}}\right)^2 < \infty,$ almost surely. 

Therefore, since we assumed that $U_{n,1} \longrightarrow \infty$ as $n \to \infty$, we get 
\be \label{Eq: Upper_bound_2_on_T_n}
T_n \le \frac{1}{\sqrt{U_{n,1}}}\displaystyle \sum_{i \ge n} \eps_i \sqrt{ 
		\frac{\sigma_{i+1}}{\tau_i} \left(\frac{\tau_n}{\tau_i}\right)^{\alpha-1}} \longrightarrow 0, \text{ as } 
		n \to \infty.
\ee

From \eqref{Eq: lower_bound_on_color_2}, we get on the event $\{\theta_{n,1} \longrightarrow 0\}$, 
\ben
C \le C' \theta_{n,1}^{\alpha -1}+ T_n \longrightarrow 0, \text{ as } n \to \infty,
\een
which is a contradiction since $C>0$.  
\end{proof}

\appendix
\section{Stochastic Approximation}\label{sec:app}
In this section we will introduce relevant quantities and results for the stochastic approximation techniques which are relevant for our model. We will use notation from \cite{Higham, Ben2}.
Define the following recursive scheme on a filtered probability space $(\Omega, \mathcal{F}, (\mathcal{F}_n)_{n\geq 0}, \mathbb{P})$ with values in a compact set $\Gamma \subset \mathbb{R}^d$:
\begin{equation}\label{def:theta}
\theta_{n+1} =  \theta_n+\frac{\sigma_{n+1}}{\tau_{n+1}}\left[ h(\theta_n)+ \frac{1}{\sigma_{n+1}}\Delta M_{n+1} \right],
\end{equation}
where $(\frac{\sigma_{n}}{\tau_{n}})_{n\geq 0}$ is a positive sequence of numbers, $h: \Gamma \rightarrow \mathbb{R}^d$ is a continuous function, $(\Delta M_n)_{n\geq 0}$ a $(\mathcal{F}_n)_{n\geq 0}$-martingale increment, defined in \eqref{def:Delta_Mn}.
The mean-field function $h$ was defined as
\begin{equation}\label{def:hh}
h(y)= \frac{f_d(y)}{\| f_d(y)\|_1}-y
\end{equation}
for $y\in \Gamma$.

The main idea of stochastic approximation techniques is to show that the recursive system defined in \eqref{def:theta} behaves like an ODE perturbed by a small noise term. Then the recursive system is shown to converge towards the equilibrium points of the deterministic flow induced by the driving term $h$.\\ 

Call $L(\{w_n\}_{n\geq 1})$ the \textit{limit set} of the sequence $\{w_n \}_{n\geq 1}$ is the set of points $x\in \mathbb{R}^d$ such that there exists s subsequence $(n_k)_{k\geq 1}$ such that $\lim_{k\rightarrow \infty} n_k=\infty$ for which $\lim_{k\rightarrow \infty} w_{n_k} =x$.

Consider furthermore the ODE 

\begin{equation}\label{def:ODE}
\dot y = h(y) 
\end{equation}
and associated to $h$ look at $(\Phi(t, x))_{t\geq 0, x \in \Gamma}$ the $\Gamma$-valued \textit{flow} of the system. The family $\{ \Phi_t\}_{t\in \mathbb{R}_+}$, where $\Phi_t(x) = \Phi(t,x)$, satisfies the group property. $\Phi_0=$Id and for all $(t,s) \in \mathbb{R}^2$
\[
\Phi_{t+s} = \Phi_t \circ \Phi_s.
\] 
For every $x\in \Gamma$, $(\Phi(t,x))_{t\geq 0}$ is the unique solution to the ODE \eqref{def:ODE}
\begin{equation}\label{def:flow}
\frac{\textrm{d}}{ \textrm{dt}} \Phi_t(x) = h(\Phi_t(x)).
\end{equation}
It exists since $h$ is locally Lipschitz which follows from continuity of the function $f$. Denote by $\mathcal{E}(\Phi) = \{ p\in \Gamma: \Phi_t(p)=p \text{ for all } t>0 \}$ the equilibrium set for the flow $\Phi$.
\begin{definition}
A compact subset $A \subset \Gamma$ is called internally chain recurrent if and only if for each $x \in A$ is chain recurrent for the flow $\Phi$ restricted to $A$. The point $x$ is called chain recurrent in $A$ if and only of for all $\delta >0,$ and $T>0,$ there exist $k\in \mathbb{N}$ and points $y_0,...,y_{k-1} \in A$ and $t_1,...,t_{k-1}$ such that for all $i=0,...,k-1$ and $y_k=x$
\[
t_i \geq T; \, \, \, \,  d(y_0,x) < \delta; \, \, \, \, d(\Phi_{t_i}(y_i),y_{i+1}) < \delta
\]
where $d(\cdot, \cdot)$ denotes the Euclidean distance in $\mathbb{R}^d$. The set of all chain recurrent points of the flow $\Phi$ will be denoted by CR$(\Phi)$.
\end{definition}
The omega limit set of $\Phi$ is defined by $\mathcal{L}(\Phi) = \bigcup_{x\in \Gamma} \omega(x)$, where 
\[
\omega(x) = \left\{ p\in \Gamma: \exists (t_k)_k \text{ such that } t_k\rightarrow \infty \, \text{ and } \, p=\lim_{k\rightarrow \infty}\Phi_{t_k}(x) \right \}.
\]
One has that 
\begin{equation}\label{sets}
\mathcal{E}(\Phi) \subset \mathcal{L}(\Phi) \subset CR(\Phi), 
\end{equation}
see page 21 of \cite{Ben2}.

We will need the notion of a strict Lyapunov function for the flow $\Phi$, see also Section 3 of \cite{Higham}.
\begin{definition}
A strict Lyapunov function for the flow $\Phi$ (or $h$) on $\Gamma$ is a continuous map $V:\mathbb{R}^d \rightarrow \mathbb{R}$ such that  $t\mapsto V(\Phi_t(x))$ is constant for $x\in \Gamma$ and strictly decreasing for all $x\in \mathbb{R}^d\setminus \Gamma$ for all $t>0$.
If such a $V$ exists we call $h$ a gradient-like vector field.
\end{definition}
\begin{lem}\label{lem:Lya}
Let the flow $\Phi$ be defined in \eqref{def:flow} and $h$ in \eqref{def:hh} and $F: [0,1]^d \rightarrow \mathbb{R}$ by
\[
F(y) = \sum_{i=1}^d \int_0^{y_i} \frac{f(z_i)}{z_i} dz_i.
\]
Then $V=-F$ is a strict Lyapunov function for $\Phi$ on $\mathcal{E}(\Phi)$.
\end{lem}

\begin{proof}
For $i=1,...,d,$ write the ODE \eqref{def:ODE} as
\begin{equation}\label{eq:rep}
\begin{split}
\dot y_i(t) &=  \frac{y_i(t)}{\|f_d(y(t)) \|_1} \left (\frac{f(y_i(t))}{y_i(t)} -\sum_{j=1}^d y_j(t) \frac{f(y_j(t))}{y_j(t)} \right )\\
& = \frac{y_i(t)}{\|f_d(y(t)) \|_1}  \left (\frac{\partial }{\partial y_i} F (y(t))  -\sum_{j =1}^d y_j(t) \frac{\partial }{\partial y_j} F (y(t))  \right )
\end{split}
\end{equation}
where $F:\mathbb{R}^d \rightarrow \mathbb{R}$ and $F(y)= \sum_{i=1}^d F_i(y_i)$. We have that $F(0)=0$ and
\[
\frac{\partial }{\partial y_i} F (y) = \frac{f(y_i)}{y_i}. 
\]
For each coordinate $i=1,\cdots, d$ we can write,
\[
F_i(y_i) = \int_0^{y_i} \frac{\partial}{\partial z} F_i(z) dz = \int_0^{y_i} \frac{f(z)}{z} dz
\]
where $z\in \mathbb{R}$.  The function $-F$ is strictly Lyapunov for $\Phi$ since
\[
\begin{split}
\frac{\textrm{d}}{\textrm{dt}} F(y(t)) & = \sum_{i=1}^d \dot y_i(t) \frac{\partial}{\partial y_i} F(y(t)) \\
& = \frac{1}{\| f_d(y) \|_1} \left \{ \sum_{i=1}^d y_i(t) \left ( \frac{\partial}{\partial y_i} F(y(t)) \right )^2 - \left ( \sum_{i=1}^d y_i(t) \frac{\partial}{\partial y_i} F(y(t)) \right )^2\right \} >0
\end{split}
\] 
where the last step follows from Jensen's inequality. Finally note that it is constant on the equilibrium set $\mathcal{E}(h)=\mathcal{E}(\Phi)$:
\[
\begin{split}
& \sum_{i=1}^d y_i(t) \left ( \frac{\partial}{\partial y_i} F(y(t)) \right )^2 - \left ( \sum_{i=1}^d y_i(t) \frac{\partial}{\partial y_i} F(y(t)) \right )^2 \\
& = \sum_{i=1}^d \frac{f^2(y_i(t))}{y_i(t)} - \left ( \sum_{i=1}^d f(y_i(t))\right )^2 \\
& = \sum_{i=1}^d \|f_d(y(t)) \|_1 f(y_i(t)) -  \|f_d(y(t)) \|_1^2 =0,
\end{split}
\]
where in the before last equality we used \eqref{Eq: deducing_equilibrium_points}.
\end{proof}

\begin{prop}\label{prop:3.2}
Let $\Phi$ be the flow on $\Gamma$, where $\Gamma$ is compact. Furthermore, let $\Lambda \subset \Gamma$ be compact invariant set and  $V:\Gamma \rightarrow \mathbb{R}$ a  Lyapunov function for $\Phi$ on $\Lambda$ with finite $V(\Lambda)$, then 
\[
CR(\Phi) \subset \Lambda.
\]
\end{prop}

The last proposition is same as Proposition 1.2 from \cite{Higham}. Choosing $V=-F$ and $\Lambda= \mathcal{E}(\Phi)$ in the last proposition together with \eqref{sets}  implies that
\[
\mathcal{E}(\Phi) = \mathcal{L}(\Phi) = CR(\Phi).
\]

The following theorem is equivalent to Theorem 6 from \cite{Kaur}, based on Theorem 5.7 from \cite{Ben2}. 
\begin{theorem}\label{th:1.2}
Let $\{ \theta_n \}_{n\geq 1}$ be a solution to \eqref{def:theta} and assume that
\begin{itemize}
\item[(i)] $\sum_{n=1}^{\infty} \left ( \frac{\sigma_n}{\tau_n}\right )^2 <\infty$.
\item[(ii)] $h$ is a local Lipschitz function.
\item[(iii)] $\sup_{n\geq 1} \mathbb{E} (\| \Delta M_{n} \|^2_2|\mathcal{F}_n) <\infty$ a.s.  
\end{itemize}
Then almost surely $L(\{ \theta_n \}_{n\geq 1})$ is a connected set, internally chain-recurrent for the flow $\Phi$ induced
by $h$, $L(\{ \theta_n \}_{n\geq 1}) \subset CR(\Phi)$.
\end{theorem}
\begin{proof}
The first assumption is trivially satisfied, the second follows from the fact that $f$ is locally Lipschitz since it is $C^1((0,1))$ and the third follows trivially from writing the martingale difference as in \eqref{def:Delta_Mn}.
\end{proof}
The next corollary is analogous to Corollary 3.3. from \cite{Higham}, the proof is adapted.
\begin{cor}\label{cor:conveqpoints}
Let $\{ \theta_n \}_{n\geq 1}$ be defined in \eqref{def:theta}. Then $\theta_n \rightarrow \theta^*$ a.s. as $n\rightarrow \infty$ where $\theta^* \in \mathcal{E}(h)$.
\end{cor}
\begin{proof}
Let $\Phi$ be the flow induced by $h$ defined in \eqref{def:hh} and let $\Gamma=L(\{\theta_n \}_{n\geq 1})$ the set of all almost sure limiting points of 
$\{\theta_n \}_{n\geq 1}$. By Theorem \ref{th:1.2} is a connected set, internally chain-recurrent for the flow $\Phi$ induced
by $h$. By Lemma \ref{lem:Lya} we know that there exists a Lyapunov function for the flow on $\Gamma$ and by Proposition \ref{prop:3.2} the set  $L(\{\theta_n \}_{n\geq 1})$ consists of equilibria, $\mathcal{E}(\Phi) =\mathcal{E}(h)$. Since they are isolated (Lemma \ref{lem:eqpoints}), $L(\{\theta_n \}_{n\geq 1})$ is an equilibrium.
\end{proof}

\section*{Acknowledgement}
The authors are grateful to V.S. Borkar for his help with stochastic approximations. The authors are grateful to Andrew Wade for his valuable insight with the class of reinforcement functions.

\end{document}